\documentclass{article}
\title{Improved rates of convergence for the multivariate Central Limit Theorem in Wasserstein distance}
\author{Thomas Bonis}
\RequirePackage{mathtools,amsmath,amsfonts,amssymb,amsthm}
\RequirePackage{hyperref}
\hypersetup{
  pagebackref=false,
  pdfborder={0 0 0},
  pdfstartview={FitH},
  hypertexnames=false,
  breaklinks,
}

\newtheorem{theorem}{Theorem}
\newtheorem{proposition}{Proposition}
\newtheorem{remark}{Remark}
\newtheorem{corollary}{Corollary}

\newtheorem{lemma}{Lemma}

\newcommand{\beq}{\begin{equation}}
\newcommand{\eeq}{\end{equation}}

\newcommand{\RR}{\mathbb{R}}
\newcommand{\NN}{\mathbb{N}}
\newcommand{\EE}{\mathbb{E}}
\newcommand{\test}{\phi}

\newcommand{\gap}{\beta}

\begin{document}
\maketitle

\begin{abstract}
We provide new bounds for the rate of convergence of the  multivariate Central Limit Theorem in Wasserstein distances of order $p \geq 2$. In particular, we obtain what we conjecture to be the asymptotically optimal rate whenever the density of the summands admits a non-zero continuous component and has a non-zero third moment.
\end{abstract}

\section{Introduction and main result}

Let $X_1,\dots,X_n$ be i.i.d. random variables drawn from a measure $\mu$ on $\RR^d$ and such that $\EE[X_1] = 0$ and $\EE[X_1 X_1^T] = I_d$. By the Central Limit Theorem, we know that the measure $\mu_n$ of $S_n = \frac{1}{\sqrt{n}} \sum_{i=1}^n X_i$ converges to the $d$-dimensional standard normal distribution $\gamma$. In this work, we wish to quantify this convergence for the family of Wasserstein distances of order $p \geq 2$, defined between any two measures $\nu$ and $\nu'$ on $\RR^d$ by
\[
W_p(\nu,\nu')^p = \inf_\pi \int_{\RR^d \times \RR^d} \|y-x\|^p \, d \pi(x,y),
\]
where $\pi$ has marginals $\nu$ and $\nu'$ and $\|\cdot\|$ is the traditional Euclidean norm.

In recent years, multiple works provided non-asymptotic bounds for $W_p(\mu_n, \gamma)$. For instance as long as $\EE[\|X_1\|^4] < \infty$, Theorem 1 \cite{Bonis} gives
\beq
\label{eq:BonisRes}
W_2(\mu_n, \gamma) \leq C \sqrt{\frac{\sqrt{d}  \|\EE[X_1 X_1^T \|X_1\|^2]\|_{HS}}{n}},
\eeq
where $C > 0$ and $\|\cdot\|_{HS}$ denotes the Hilbert-Schmidt norm. Similar results were also obtained for other $W_p$ distances \cite{Bonis, ModerateDeviations}. However, this bound is not sharp with respect to the dimension. Indeed, if $X_1$ has i.i.d. components, (\ref{eq:BonisRes}) scales with $d^{3/4}$ while an optimal bound would scale with $\sqrt{d}$.
Sharper bounds have been obtained under additional assumptions on the measure $\mu$. For instance, if $\mu$ satisfies a Poincaré inequality with constant $K \geq 1$, Theorem 4.1 \cite{Fathi1} gives 
\beq
\label{eq:Fathi}
W_2(\mu_n, \gamma) \leq C \sqrt{\frac{(K-1)d}{n}}
\eeq
and similar results have been obtained for any $W_p$ distances with $p \geq 1$ in Theorem 1.2 \cite{FangKoike2023} under the additional assumption that $\mu$ is log-concave. As a consequence of (\ref{eq:Fathi}), if $\mu$ is log-concave then it admits a Poincar\'e constant $K \leq C \sqrt{\log d}$ for some $C >0$ \cite{klartag} and if the Kannan-Lovász-Simonovits isoperimetric conjecture is true then $K \leq C$. Finally, for uniformly log-concave measures, the optimal dependency on  $\sqrt{d}$ is obtained in Theorem 3.4 \cite{Fathi2} without any further assumptions. 

Some insight on the conditions required to obtain this optimal dependency on the dimension in a more general case can be obtained from Proposition 1.2 \cite{Zhai} which states that, if $X_1$ takes value in the lattice $h \mathbb{Z}^d$ with $h>0$, then 
\[
\liminf\limits_{n \rightarrow \infty} \sqrt{n} W_2(\mu_n, \gamma) \geq \frac{\sqrt{d}h}{4}.
\]
In particular, if $h$ is of order $\sqrt{d}$ then 
$\liminf\limits_{n \rightarrow \infty} \sqrt{n} W_2(\mu_n, \gamma) \geq C d$. Therefore, if one wants $W_2(\mu_n, \gamma)$ to scale with the square root of the dimension, one would require $h$ to be independent of $d$, or $X_1$ to not be lattice-distributed. Such a result does not come as surprising in the light of known asymptotic results obtained in the univariate setting. Indeed, according to Theorem 1.2 \cite{Rio2}, if $X_1$ takes values in $\{a + kh \mid k \in \mathbb{Z} \}$ for some $a \in \mathbb{R}, h > 0$ and has a finite moment of order $p+2$ with $p \in ]1, 2]$, then 
\beq
\label{eq:optimal1}
\liminf\limits_{n \rightarrow \infty} \sqrt{n} W_p(\mu_n, \gamma) = \frac{1}{6} \|\EE[X_1^3](Z^2 - 1) + h U\|_p,
\eeq
where $Z \sim \gamma, U$ is a uniform random variable on $[-1/2, 1/2]$ independent of $Z$ and $\|\cdot\|_p = \EE[\|\cdot\|^p]^{1/p}$. On the other hand, as long as $X_1$ is not lattice-distributed, one has 
\beq
\label{eq:optimal2}
\liminf\limits_{n \rightarrow \infty} \sqrt{n} W_p(\mu_n, \gamma) = \frac{1}{6} \|\EE[X_1^3](Z^2 - 1)\|_p.
\eeq
Furthermore, faster rates of convergence have been obtained for all $p \geq 1$ whenever the first moments of $\mu$ and $\gamma$ are equal and $\mu$ satisfies the Cramer's condition \cite{Bobkov2018}. 

 One can thus expect the rate of convergence for the central limit theorem in Wasserstein distance in a high-dimensional setting to not only be determined by the moments of $X_1$ but to also depend on whether the measure is lattice-distributed. In other words, along with the large-scale behaviour of $\mu$, described by its moments, we expect a tight bound on $W_p(\mu_n, \gamma)$ to include a term corresponding to the small-scale behaviour of $\mu$. In this work, we provide a first instance of such a result in the multidimensional setting. In particular, we obtain the following asymptotic bound.
\begin{corollary}
\label{thm:mainCor}
Let $p \geq 2$ and $X_1,\dots,X_n$ be i.i.d. centered random variables drawn from a measure $\mu$ on $\RR^d$ with identity covariance matrix and finite moment of order $p+2$. Suppose there exists $h > 0$ such that the matrix 
\[
\EE[(X_2 - X_1) (X_2 - X_1)^{T} 1_{\|X_2 - X_1\| \leq h}]
\]
is positive-definite. Then, the measure $\mu_n$ of $S_n = \frac{1}{\sqrt{n}} \sum_{i=1}^n X_i$ verifies
\[
\sqrt{n} W_p(\mu_n, \gamma) \leq \frac{1}{6} \|\EE[X_1^{\otimes 3}] (Z^{\otimes 2} - I_d)\|_p + C p h \sqrt{ d} + C_{d,p,\mu} \mathcal{O} \left(\log(n)^{-1/p}\right),
\]
where $C > 0$ is a generic constant, $C_{d,p, \mu}$ is a constant depending on $d,p$ and $\mu$, $Z$ is drawn from the $d$-dimensional standard normal distribution $\gamma$ and $\EE[X_1^{\otimes 3}] (Z^{\otimes 2} - I_d)$ is a vector whose $i$-th coordinate is given by
\[
 (\EE[X_1^{\otimes 3}] (Z^{\otimes 2} - I_d))_i \coloneqq  \sum_{j,k} \EE[(X_1)_i (X_1)_j (X_1)_k] (Z_j Z_k - 1_{j=k}).
\]
Furthermore, if $\mu$ has a non-zero absolutely continuous component with respect to the Lebesgues measure then,
\beq
\label{eq:mybound}
\sqrt{n} W_p(\mu_n, \gamma) \leq \frac{1}{6} \|\EE[X_1^{\otimes 3}] (Z^{\otimes 2} - I_d)\|_p + C_{d,p,\mu}  o\left(1\right).
\eeq
\end{corollary}
Whenever $\mu$ admits a non-zero continuous component and has a non-zero third moment, we conjecture our result to be asymptotically optimal as it is a natural multidimensional generalization of (\ref{eq:optimal2}). In particular, if $X_1$ has i.i.d. components we recover the correct dependency on $\sqrt{d}$ since, by Lemma~\ref{lem:HermiteNorm},
\[
\sqrt{n} W_p(\mu_n, \gamma) \leq \frac{\EE[(X_1)_1^3] \sqrt{(p-1)d} }{3 \sqrt{2}}  + C_{d,p,\mu}  o\left(1\right).
\]
Our bound is also asymptotically sharper than known existing bounds. 
Indeed, using Lemma~\ref{lem:HermiteNorm} and Lemma~\ref{lem:CS}, we obtain 
\[
\|\EE[X_1^{\otimes 3}] (Z^{\otimes 2} - I_d)\|_p \leq  \sqrt{ 2 (p-1) \sqrt{d}  \|\EE[X_1 X_1^T \|X_1\|^2]\|_{HS}},
\]
thus recovering (\ref{eq:BonisRes}) in the asymptotic setting. In particular, this means that if $\|X_1\| \leq M$ almost surely then, for any $p \geq 1$, 
\[
\sqrt{n} W_p(\mu_n, \gamma) \leq M \sqrt{\frac{(p-1)d}{18}} + C_{d,p,\mu}  o\left(1\right).
\]
Remark that this bound scales with at least $d$ as $M$ must be of order at least $\sqrt{d}$. On the other hand, if $\mu$ admits a Stein kernel $\tau$ as defined in \cite{Stein}, combining Lemmas~\ref{lem:HermiteNorm} and \ref{lem:third} gives
\[
\|\EE[X_1^{\otimes 3}] (Z^{\otimes 2} - I_d)\|_p  \leq 
 2 \sqrt{2 (p-1) \EE[\|\tau - I_d\|^2]}. 
\]
Hence, following the work of \cite{Fathi1}, if $\mu$ admits a Poincaré constant $K \geq 1$ we can generalize (\ref{eq:Fathi}) to all $p \geq 1$:
\[
\sqrt{n} W_p(\mu_n, \gamma) \leq \frac{\sqrt{2 (p-1) (K-1)d}}{3} + C_{d,p,\mu}  o\left(1\right).
\]
Let us note that, asymptotically, this bound depends only on $\sqrt{p-1}$, thus improving on the bound obtained in Theorem 1.2 \cite{FangKoike2023} which scales with $p^2$ while lifting the requirement for $\mu$ to be log-concave.

For lattice-distributed measures, our bound is close to matching a multidimensional equivalent of (\ref{eq:optimal1}) but still requires improvements. However, obtaining the optimal rate of convergence for discrete but non lattice-distributed random variables is still an open issue. In any case, let us note that the remainder term is likely sub-optimal.

Corollary \ref{thm:mainCor} is derived from a non-asymptotic bound obtained in Theorem~\ref{thm:mainCLT} which also deals with non-identically distributed random variables.
Our result is derived through refinements on a variant of Stein's method used in \cite{Bonis} which might be of interest in other contexts. 

\paragraph{Acknowledgements}

I would like to thank the anonymous referee for their many comments which helped improve this paper. I am also grateful to Olivier Gu\'edon for pointing out \cite{klartag} to me. 

\section{Notations}
\label{sec:notations}
Let $d$ be a positive integer. 
For any $k \in \NN$, let $(\RR^d)^{\otimes k}$ be the set of elements of the form $(x_j)_{j \in \{1,\dots,d\}^k} \in \RR^{d^k}$.
For $x \in \RR^d$ and $k \in \NN$, we denote by $x^{\otimes k}$ the element of $(\RR^d)^{\otimes k}$ such that 
\[
\forall j \in \{1,\dots,d\}^k, (x^{\otimes k})_{j} = \prod_{i=1}^k x_{j_i}.
\]
For any $x,y \in (\RR^d)^{\otimes k}$, we denote by $<x,y>$ the Hilbert-Schmidt scalar product between $x$ and $y$ defined by 
\[
<x,y> = \sum_{i \in \{1,\dots,d\}^k}  x_i y_i,
\]
and, by extension, we write 
\[
\|x\|^2 = <x,x>.
\]
Furthermore, for any $x \in (\RR^d)^{\otimes (k+1)}$ and $y  \in (\RR^d)^{\otimes k}$, let $xy$ be the vector defined by  
\[
\forall i \in \{1,\dots,d\}, (xy)_i = \sum_{j \in \{1, \dots, d\}^k} x_{i,j} y_j. 
\]
%For any $A \in \RR^{\otimes d_1}$ and $B \in \RR^{\otimes d_2}$ with $d_2 \geq d_1$, let %$C \in \RR^{\otimes d_2 - d_1}$ where
%\[
%\forall i \in \{1, \dots, d\}^{d_2 - d_1}, C_i = \sum_{j \in \{1, \dots, d\}^{d_1}} A_j %B_{i,j},
%\]
%where $B_{i,j} = B_{i_1,\dots,i_{d_1}, j_1, \dots, j_{d_2 - d_1}}$.
%Finally, for any $A \in \RR^{\otimes 3}$, let 
%\[
%\|A\|_{op} = \sup_{B \in \RR^{\otimes 2}, \|B\| = 1} \|AB\|.
%\]
 For any $k \in \NN$, any function $\test$ with partial derivatives of order $k$ and any $x \in \RR^d$, we denote by $\nabla^k \test (x)\in (\RR^{d})^{\otimes k}$ the $k$-th gradient of $\test$ at $x$:
\[
\forall j \in \{1,\dots,d\}^k, (\nabla^k \test (x))_j = \frac{\partial^k \test}{\partial x_{j_1} \dots \partial x_{j_k} }(x).
\]
For any $k \in \NN$, let $H_k$ be the $d$-dimensional Hermite polynomial, defined by
\[
\forall x \in \RR^d, H_k(x) = (-1)^k e^{\frac{\|x\|^2}{2}} \nabla^k e^{- \frac{\|x\|^2}{2}}.
\] 
Finally, for any random variable $X$ on $\mathbb{R}^d$, we denote by $\|X\|_p$ the $L_p$-norm of $X$, that is 
\[
\|X\|_p \coloneqq \EE[\|X\|^p]^{1/p}.
\]

\section{Main Result}
\label{sec:result}
Let $n > 0$ and $W_1, \dots, W_n$ be independent centered random variables on $\mathbb{R}^d$ such that $W = \sum_{i=1}^n W_i$ has identity covariance matrix and $\max_{i \in \{1, \dots, n\}} \|\EE[W_i^{\otimes 2}]\| < 1$. We denote by $\nu$ the measure of $W$.
For any $i \in \{1, \dots, n\}$, let $D_i = W'_i - W_i$, where $W'_i$ is an independent copy of $W_i$. 
Let us define a set of features describing the large-scale behaviour of the variables $(W_i)_{1 \leq i \leq n}$:
\begin{itemize}
\item  $\forall i \in \{1,\dots,n\}, \xi_i = -\log(\|\EE[W_i^{\otimes 2}]\|)$;
\item $\forall q > 0, L_q = \sum_{i=1}^n \EE[\|W_{i}\|^{q}]$;
\item $\forall q > 2, N_q = \sum_{i=1}^n \frac{1}{\xi_i}\EE[\|D_{i}\|^{q} (1_{\|D_i\|^2 \geq  \xi_i \|\EE[W_i^{\otimes 2}]\|^{2/3}} + \xi_i^{-1})^{q/2 - 1}]$; 
\item $N'_4 = \sum_{i=1}^n \frac{\|\EE[D_i^{\otimes 2} \|W_i\|\|D_i\|]\|}{\sqrt{\|\EE[W_i^{\otimes 2}]\| \xi_i}}$.
\end{itemize}
Now, for any $\beta > 0$, let $D_{i, \beta} = D_i 1_{\|D_i\| \leq \beta}$. If $\EE\left[\sum_{i=1}^n D_{i, \beta}^{\otimes 2}\right]$ is positive-definite, we consider the following small-scale feature:
\[
\forall q \geq 0, \beta_q =   \sum_{i=1}^n \EE\left[ \left\|\EE\left[\sum_{i=1}^n D_{i, \beta}^{\otimes 2}\right]^{-1} D_{i, \beta} \right\|^q \right].
\]
\begin{theorem}
\label{thm:mainCLT}
Let $p \geq 2$ such that $p \leq \min_{i \in \{1, \dots, n\}} \|\EE[W_i^{\otimes 2}]\|^{-1}$ and suppose $L_{p+2} < \infty$. Let $\epsilon \coloneqq  \max_{i \in \{1, \dots, n\}} \|\EE[W_i^{\otimes 2}]\|^{2/3}$. If there exists $0 < \gap < \sqrt{\epsilon}$ such that $\EE\left[\sum_{i=1}^n D_{i, \beta}^{\otimes 2}\right]$ is positive-definite, then, for any $q, r > p$ such that $\frac{1}{q} + \frac{1}{r} = \frac{1}{p}$, we have 
\begin{multline*}
W_p(\nu,\gamma) \leq \frac{\|\EE[W^{\otimes 3}] H_2(Z)\|_p}{6} + C  \left(\gap p \sqrt{d}    + \frac{ (2r-1)^{3/2} \|\EE[W^{\otimes 3}]\| W_{q}(\nu,\gamma)}{\sqrt{(p-1)\epsilon}}\right) \\ +  C p \left( \epsilon \left( \sqrt{p}( \sqrt{\beta_2} + \sqrt{d}) + p \left(\beta_p^{1/p} + L_p^{1/p}\right) \right)
+  \sqrt{ N_4} 
+ (p N_{p+2})^{1/p}
+ N'_4 \right).
\end{multline*}
\end{theorem}
In order to prove Corollary~\ref{thm:mainCor} from this result, we take 
\[
\forall i \in \{1, \dots, n\}, W_i = \frac{X_i}{\sqrt{n}}
\]
and $\beta = \frac{h}{\sqrt{n}}$ and let us assume $n$ is sufficiently large so that $\beta < \frac{\|\EE[X_1^{\otimes 2}]\|^{2/3}}{n^{2/3}} = \frac{d^{1/3}}{n^{2/3}}$. In the following, we denote by $C$ a positive constant depending on properties of $\mu$ but independent of $n$. 
First, we have 
\[
\sqrt{p}(\sqrt{\beta_2} + \sqrt{d}) + p  \left(\beta_p^{1/p} + L_p^{1/p} \right) \leq C
\]
and, since $\epsilon =   \frac{C}{n^{2/3}}$,
\[
\epsilon \left(\sqrt{p}(\sqrt{\beta_2} + \sqrt{d}) + p  \left(\beta_p^{1/p} + L_p^{1/p} \right)\right) \leq \frac{C}{n^{2/3}}.
\]
Then, 
\[
N'_4 = \frac{C}{\sqrt{n \log(n)}}
\]
and, since we have $\lim\limits_{n \rightarrow \infty} n \epsilon \xi_1= +\infty$, 
\[
p \sqrt{N_4} + p^{1+1/p} N_{p+2}^{1/p} = o\left(\frac{1}{\sqrt{n} \log(n)^{1/p}}\right).
\]
Furthermore, since $X_1$ has finite moment of order $p+2$, we can use Theorem~6 from \cite{Bonis} to obtain 
\[
W_{p+1/2}(\mu_n, \gamma) \leq  \frac{C}{n^{1/2 - 1/4p}}  \leq  \frac{C}{n^{3/8}}.
\]   
Thus, since $\|\EE[W^{\otimes 3}]\| = \frac{C}{\sqrt{n}}$,
\[
 \frac{\|\EE[W^{\otimes 3}]\| W_{p+1/2}(\mu_n,\gamma)}{\sqrt{\epsilon}}  \leq  \frac{C}{n^{13/24}}.
\]
which concludes the proof whenever $\mu$ does not admit an absolutely continuous component with respect to the Lebesgues measure. If it does, let us denote by $\mu_c$ this continuous component.
For any $h > 0$, there must exist a ball $\mathcal{B}$ with radius $h$ and non-zero mass for $\mu_c$. Remark that  
\[
\int_{\mathcal{B}^2} (x' - x)^{\otimes 2} d\mu_c(x) d\mu_c(x')
\]
must be positive definite. Otherwise the dimension of the support of $\mu_c$ on this ball would be lower than $d$ which is impossible since $\mu_c$ is absolutely continuous with respect to the Lebesgues measure. Thus, 
\[
\forall h, q > 0, \beta_q(h) = \frac{1}{n^{q/2-1}} \EE\left[ \left\|\EE\left[(X'_i - X_i)^{\otimes 2} 1_{\|X'_i - X_i\| \leq h}\right]^{-1} (X'_i - X_i) 1_{\|X'_i - X_i\| \leq h} \right\|^q \right]
\]
must be finite. 
Therefore, for $n$ sufficiently large, we can take $h_n$ as 
\[
h_n = \inf \left\{ h \geq \frac{1}{n} : \sqrt{\beta_2(h)} + \beta_p(h)^{1/p} \leq n^{1/7} \right\}.
\]
Then $\lim\limits_{n \rightarrow \infty} h_n = 0$ and 
\[
\epsilon \left(\sqrt{p}(\sqrt{\beta_2(h_n)} + \sqrt{d}) + p  \left(\beta_p(h_n)^{1/p} + L_p^{1/p} \right)\right) = O\left(n^{1/7 - 2/3} \right) = o\left(\frac{1}{\sqrt{n}} \right),
\]
which yields the desired result. 

\begin{remark}
Note that we restricted ourselves to the existence of a moment of order $p+2$ for the summands to simplify computations. Let us note that one could only consider existence of a moment of order $p + l$ with $l < 2$ only in order to obtain the rate $o\left(n^{-1/2 + 1/p - l/2p} \log(n)^{-l/2p} \right)$ for the i.i.d. case which would slightly improve on Theorem 6 \cite{Bonis} in which the rate $O\left(n^{-1/2 + 1/p - l/2p} \right) $ was obtained.
Our approach would also be able to deal with varying moment assumption where each variable $W_i$ admits a finite moment of order $p + l_i$ for non identically distributed summands.
\end{remark}

\section{Diffusion interpolation approach}

Let $p> 0$ and $W$ be a random variable drawn from a measure $\nu$ on $\RR^d$. In the following, we assume $\nu$ admits a density $h$ with respect to the Gaussian measure which is both bounded and with bounded gradient. These additional assumptions can later be lifted to obtain Theorem~\ref{thm:mainCLT} using approximation arguments similar to those developed in Section 8 \cite{Bonis}.

Let $t > 0$ and let us consider the random variable $F_t := e^{-t} W + \sqrt{1 - e^{-2t}} Z$, where $Z$ is a random variable drawn from the $d$-dimensional standard Gaussian measure $\gamma$ and independent of $W$. We denote by $\nu_t$ the measure of $F_t$. Due to our assumptions on $h$, $\nu_t$ admits a smooth density $h_t$ with respect to $\gamma$. We can thus consider the score function of $F_t$ defined by 
\[
\rho_t \coloneqq \nabla \log h_t (F_t). 
\]
Then, by Equation (3.8) \cite{Stein}, we have 
\[
W_p(\nu, \nu_t) \leq \int_0^t \|\rho_t\|_p \, dt
\]
and, since $\lim\limits_{t \rightarrow \infty} \nu_t = \gamma$,
\[
W_p(\nu, \gamma) \leq \int_0^\infty \|\rho_t\|_p \, dt.
\]
We are thus left with bounding $\|\rho_t\|_p$ for all $t \geq 0$. 

One can first remark that this score function verifies the following formula (see e.g.  Lemma IV.1 \cite{NPS2014}):
\beq
\label{eq:ScoreFunction}
\rho_t = e^{-t} \EE\left[W - \frac{Z}{\sqrt{\Delta(t)}} \mid F_t\right] \text{ a.s.},
\eeq
where $\Delta(t) \coloneqq e^{2t}-1$. A first, somewhat trivial, bound on $\|\rho_t\|_p$ can then be obtained by applying Jensen's and the triangular inequalities:
\beq
\label{eq:Psi1}
\|\rho_t\|_p \leq e^{-t} \left(\|\EE[W \mid F_t]\|_p + \frac{\|\EE[ Z \mid F_t]\|_p}{\sqrt{\Delta(t)}}\right) \leq e^{-t} \left(\|W\|_p + \frac{\|Z\|_p}{\sqrt{\Delta(t)}}\right).
\eeq
Note that this bound can still be nearly optimal for small values of $t$. Indeed, whenever $W$ takes values in $h \mathbb{Z}^d$, one has, for small enough values of $t << h$, 
\[
W_2(\nu, \nu_t) \approx \|F_t - W\|_2 = (1 - e^{-t}) \left\|W\right\|_2 + \sqrt{1-e^{-2t}} \left\|Z\right\|_2 = \int_0^t e^{-t} \left(\|W\|_2 + \frac{\|Z\|_2}{\sqrt{\Delta(t)}}\right) \, dt.
\]
However, for continuous measures $\nu$ or for higher values of $t$, it is usually possible to obtain better bounds on $\|\rho_t\|_p$. For instance, (\ref{eq:BonisRes}) is obtained by combining (\ref{eq:Psi1}) with another bound on $\|\rho_t\|_p$ which holds for large values of $t$. A similar approach was used in \cite{ModerateDeviations} to provide quantitative results for normal approximation in various frameworks such as Wiener chaos or homogeneous sums. In this work, we refine this approach by using three different bounds: (\ref{eq:Psi1}) for small values of $t$, a bound for medium values of $t$ highlighting the small-scale behaviour of the measure $\nu$ and a last bound for larger values of $t$ which depends on the large-scale structure of $\nu$ through its moments.

\section{Bounding \texorpdfstring{$\|\rho_t\|_p$}{rhot}}
\subsection{Small times}

Let $p \geq 2$ and let $W = \sum_{i=1}^n W_i$ where the $(W_i)_{1 \leq i \leq n}$ are centered and independent random variables on $\mathbb{R}^d$ with finite moment of order $p$. If $\EE[W^{\otimes 2}] = I_d$, there exists $C >0$ such that
\beq
\label{eq:small-time}
\|\rho_t\|_p \leq \Psi_1(t) \coloneqq C e^{-t} \left(\sqrt{dp}\left(1 + \frac{1}{\sqrt{\Delta(t)}} \right) +  p L_p^{1/p} \right).
\eeq
Indeed, 
since the $(W_i)_{1 \leq i \leq n}$ are independent and centered, we can use Rosenthal's inequality (see Lemma~\ref{lem:Rosenthal}) to obtain 
\[
\|W\|_p \leq C  \left(\sqrt{dp} +  p L_p^{1/p} \right).
\]
On the other hand, by Lemma~\ref{lem:HermiteNorm},
\[
\|Z\|_p \leq \sqrt{d(p-1)}.
\]
Injecting these bounds into (\ref{eq:Psi1}) then yields (\ref{eq:small-time}).

\subsection{Medium times}

 When looking at (\ref{eq:Psi1}), we can see that, for small values of $t$, the main contributor of this bound is  $\|\EE[Z\mid F_t]\|_p/\sqrt{\Delta(t)}$. In the previous Section, we upper bounded this quantity somewhat crudely by using Jensen's inequality. In this Section, we establish a sharper bound on 
 $\|\rho_t\|_p$ by proving a variant of Proposition~6.1 \cite{ModerateDeviations} leveraging the small scale features of $W$. We start by covering the more general exchangeable pair framework, a standard framework for applying Stein's method, before tackling the specific Central Limit Theorem case. 

\subsubsection{Exchangeable pairs framework}

\begin{proposition}
\label{pro:ZhangBis}
Let $p \geq 2$ and $(W,W')$ be a pair of random variables on $\mathbb{R}^d$ such that $(W, W')$ and $(W',W)$ follow the same law. For any $t \geq 0$, let $\eta_p(t) = \Delta(t) / (p-1)$ and $D_t = (W'-W) 1_{\|W'-W\|^2 \leq \eta_p(t)}$.
For any $0 < s < t$ such that $\EE[D_s^{\otimes 2}]$ is positive-definite, we have  
\[
\|\rho_t\|_p \leq  e^{-t} \left(\|\EE[\Lambda_s D_s  + W \mid W]\|_p +  \frac{C}{\sqrt{\eta_p(t)}}\|\EE[\Gamma_s \mid W] - \EE[\Gamma_s] \|_p  +  \frac{C \sqrt{d} \eta_p(s) }{\eta_p(t)^{3/2}}\right),
\]
where $C > 0$ is a generic constant, $\Lambda_s = \EE[D_s^{\otimes 2}]^{-1}$ and $\Gamma_s = \frac{1}{2} \Lambda_s D_s^{\otimes 2}$.
\end{proposition}
The proof of this result mostly follows the proof of Proposition~6.1 \cite{ModerateDeviations}. 
\begin{proof}
Let $0 < s < t$ and let 
\[
\tau_t =  \left( \Lambda_s D_s + \frac{\Gamma_s  Z}{ \sqrt{\Delta(t)}}+ \sum_{k=3}^\infty  a_k \frac{(\Gamma_s \otimes D_s^{\otimes (k-1)}) H_{k}(Z)}{\Delta(t)^{k/2}} \right),
\]
with $a_k = \frac{1}{k!} - \frac{1}{4 (k-2)!}$.
A small modification of Lemma 6.5 \cite{FangKoike2022} (see also the proof of Lemma~\ref{lem:taut}) gives 
\[
\EE[\tau_t \mid F_t] = 0.
\]
Therefore, 
\[
\rho_t = \rho_t + e^{-t} \EE[\tau_t \mid F_t]
\]
and using (\ref{eq:ScoreFunction}) along with the triangle inequality yields
\begin{multline*}
e^t \|\rho_t\|_p \leq \left\|\EE[\Lambda_s D_s + W \mid F_t] \right\|_p 
+ \frac{1}{\sqrt{\Delta(t)}} \left\|\EE\left[\left(\Gamma_s  - I_d \right) Z \mid F_t \right] \right\|_p \\
+ \sum_{k=3}^\infty \frac{a_k}{ \Delta(t)^{k/2}} \left\| \EE[ (\Gamma_s \otimes D_s^{\otimes (k-1)})H_k(Z)\mid F_t]\right\|_p.
\end{multline*}
Then, since $Z$ and $W$ are independent, we have, by Jensen's inequality,
\begin{multline*}
e^t \|\rho_t\|_p \leq \|\EE[\Lambda_s D_s + W \mid W] \|_p 
+ \frac{1}{\sqrt{\Delta(t)}} \left\|\left(\EE\left[\Gamma_s \mid W\right] - I_d \right) Z  \right\|_p\\
+ \sum_{k=3}^\infty \frac{a_k}{ \Delta(t)^{k/2}} \left\| \EE[ \Gamma_s \otimes D_s^{\otimes (k-1)}\mid W]H_k(Z)\right\|_p
\end{multline*}
and, by Lemma~\ref{lem:HermiteNorm},
\begin{multline*}
e^t \|\rho_t\|_p \leq \|\EE[\Lambda_s D_s + W \mid W] \|_p 
+ \frac{1}{\sqrt{\eta_p(t)}} \left\|\EE\left[\Gamma_s \mid W\right]  - I_d   \right\|_p\\
+ \sum_{k=3}^\infty \frac{a_k \sqrt{k!}}{ \eta_p(t)^{k/2}} \left\| \EE[ \Gamma_s \otimes D_s^{\otimes (k-1)}\mid W]\right\|_p.
\end{multline*}
Since $\Gamma_s$ is positive-definite, we have, for any $k \geq 3$, 
\begin{align*}
\left\| \EE[ \Gamma_s \otimes D_s^{\otimes (k-1)}\mid W]\right\|_p & \leq \left\| \EE[ \Gamma_s  \|D_s\|^{k-1} \mid W]\right\|_p \\
& \leq \eta_p(s)^{(k-1)/2} \left\| \EE[ \Gamma_s \mid W]\right\|_p \\
& \leq \eta_p(s)^{(k-1)/2} \left(\left\| \EE[ \Gamma_s \mid W] - I_d \right\|_p + \|I_d\|\right) \\
& \leq \eta_p(s)^{(k-1)/2} \left(\left\| \EE[ \Gamma_s \mid W] - I_d \right\|_p + \sqrt{d} \right).
\end{align*}
Thus, since $\sum_{k=3}^\infty  a_k \sqrt{k!} < \infty$ and $\eta_p(s) \leq \eta_p(t)$, we obtain that there exists $C > 0$ such that
\[
e^t \|\rho_t\|_p \leq \|\EE[\Lambda_s D_s + W \mid W] \|_p 
+ \frac{C}{\sqrt{\eta_p(t)}} \left\|\EE\left[\Gamma_s \mid W\right] - I_d   \right\|_p
+ \frac{C \sqrt{d} \eta_p(s)}{\eta_p(t)^{3/2}}.
\]
Finally, one can remark that, by definition of $\Gamma_s$, 
\[
\EE[\Gamma_s] = I_d, 
\]
concluding the proof.
\end{proof}
\subsubsection{Sum of independent variables}

\begin{proposition}
\label{pro:Psi2}
Let $W = \sum_{i=1}^n W_i$ where the $(W_i)_{1 \leq i \leq n}$ are independent random variables on $\mathbb{R}^d$ with finite moment of order $p \geq 2$. For any $i \in \{1,\dots,n\}$ and $\beta > 0$, let $D_{i, \beta} = (W'_i - W_i) 1_{\|D_i\| \leq \beta}$ where $W'_i$ is an independent copy of $W_i$. Suppose there exists $\gap > 0$ such that
\[
\Lambda_\gap^{-1} =  \sum_{i=1}^n \EE[D_{i, \beta}^{\otimes 2}]
\]
is positive-definite.
Then, for any $t$ such that $\Delta(t) \geq (p-1) \gap^2$, there exists $C > 0$ such that 
\[
\|\rho_t\|_p \leq \Psi_2(t) \coloneqq C\left(  \sqrt{p}\left(\sqrt{\gap_2} + \sqrt{L_2}\right)   + p\left(\gap_p^{1/p} + L_p^{1/p} \right)  + \frac{p^{3/2} \sqrt{d} \gap^2}{ \Delta(t)^{3/2}} \right),
\]
where 
\[
\forall q \geq 0, \beta_q =   \sum_{i=1}^n \EE\left[ \left\|\Lambda_\beta D_{i,\beta} \right\|^q \right]
\]
and 
\[
\forall q \geq 0, L_q =   \sum_{i=1}^n \EE\left[ \left\|W_i \right\|^q \right].
\]
\end{proposition}

In the following, we denote by $C$ a generic positive constant. 
Let $s$ be such that $\Delta(s) = (p-1) \gap^2$ and let $ t > s$. Let $W' = W + (W'_I - W_I)$ where $I$ is a uniform random variable on $\{1,\dots,n\}$. Since $(W,W')$ and $(W',W)$ follow the same law, we can apply Proposition~\ref{pro:ZhangBis} to obtain 
\[
e^t \|\rho_t\|_p \leq  \|\EE[\Lambda_s D_s \mid W]\|_p + \|W\|_p + \sqrt{\frac{p - 1}{\Delta(t)}} \|\EE[\Gamma_s \mid W] - \EE[\Gamma_s] \|_p + C  \frac{(p-1)^{3/2} \sqrt{d}  \gap^2}{ \Delta(t)^{3/2}},
\]
with $\Lambda_s = n\Lambda_\gap$.
First, following the proof of (\ref{eq:small-time}), we have 
\[
\|W\|_p \leq C  \left(\sqrt{pL_2} +  p L_p^{1/p} \right).
\]
% Furthermore, let us note, since $\Lambda$ is positive-definite and symmetric, there exists an orthonormal basis $e_1, \dots, e_d$ for which $\Lambda$ is diagonal. Writing coordinates in this basis, we have, 
% \begin{align*}
% \EE[\|\Lambda_\beta \sum_{i=1}^n (W'_i - W_i) 1_{\|W'_i - W_i\| \leq \beta}\|^2] & = \sum_{j=1}^D (\Lambda_\beta)_{j,j}^2 \sum_{i=1}^n \EE[ (W'_i - W_i)_j^2 1_{\|W'_i - W_i\| \leq \beta}] \\
% & = \sum_{j=1}^D (\Lambda_\beta)_{j,j} = Tr(\Lambda_\beta).
% \end{align*}
Then, by definition of $D_s$ and since $I$ is independent of $W$, 
\[
\EE[D_s \mid W] = \frac{1}{n} \sum_{i=1}^n \EE[D_{i, \beta} \mid W].
\]
Hence, 
\[
\| \EE[\Lambda_s D_s \mid W]\|_p =  \left\| \Lambda_\beta \sum_{i=1}^n \EE\left[D_{i, \beta} \mid W\right]\right\|_p
\]
and, by Jensen's inequality,
\[
\| \EE[\Lambda_s D_s \mid W]\|_p \leq \left\| \sum_{i=1}^n \Lambda_\beta D_{i, \beta}\right\|_p.
\]
Let $ i \in \{1, \dots, n\}$. Since $W'_i$ and $W_i$ are independent,
we have  
\[
\EE[D_{i,\beta}] = 0.
\]
We can thus apply Rosenthal's inequality (see Lemma~\ref{lem:Rosenthal}) to obtain 
\begin{align*}
\|\EE[\Lambda_s D_s \mid W]\|_p & \leq 
C \sqrt{p} \left(\sum_{i=1}^n  \left\|\Lambda_\beta D_{i, \beta} \right\|_2^2\right)^{1/2} 
+C  p \left(\sum_{i=1}^n  \left\|\Lambda_\beta D_{i, \beta} \right\|_p^p\right)^{1/p} \\
& \leq C \left(\sqrt{p \beta_2}  + p \beta_p^{1/p} \right).
\end{align*}
Similarly,
\[
\|\EE[\Gamma_s \mid W] - \EE[\Gamma_s] \|_p \leq 
 C \sqrt{p} \left(\sum_{i=1}^n  \left\| \Lambda_\beta D_{i,\beta}^{\otimes 2} \right\|_2^2\right)^{1/2} \\
+ C p \left(\sum_{i=1}^n  \left\| \Lambda_\beta D_{i,\beta}^{\otimes 2} \right\|_p^p\right)^{1/p}
\]
and, since $\|D_{i, \beta}\| \leq \beta \leq \sqrt{\Delta(t) / (p-1)}$,  we can use Cauchy-Schwarz inequality to obtain 
\[
\left\| \Lambda_\beta  D_{i,\beta}^{\otimes 2} \right\|  \leq  \left\| \Lambda_\beta D_{i,\beta}\right\|  \left\|D_{i,\beta}\right\| 
 \leq \sqrt{\frac{\Delta(t)}{p-1}} \left\| \Lambda_\beta D_{i,\beta} \right\|.
\]
Therefore,
\begin{align*}
 \sqrt{\frac{p - 1}{\Delta(t)}} \|\EE[\Gamma_s \mid W] - \EE[\Gamma_s] \|_p & \leq  
 C \sqrt{p} \left(\sum_{i=1}^n  \left\| \Lambda_\beta D_{i,\beta}\right\|_2^2\right)^{1/2}
+ Cp  \left(\sum_{i=1}^n  \left\| \Lambda_\beta D_{i,\beta}\right\|_p^p\right)^{1/p} \\
& \leq  C \left(\sqrt{p \beta_2} + p \beta_p^{1/p} \right)
\end{align*}
which concludes the proof. 
% Finally, by definition of $\Lambda_\beta$,
% \begin{align*}
% \beta_2 & = \sum_{i=1}^n  \sum_{j,k,l \in \{1,\dots, d\}} \EE[(\Lambda_\beta)_{j,k} (\Lambda_\beta)_{j,l} (D_{i,\beta})_k (D_{i,\beta})_l] \\
% & =  \sum_{j,k,l \in \{1,\dots, d\}} (\Lambda_\beta)_{j,k} (\Lambda_\beta)_{j,l} \EE\left[\sum_{i=1}^n (D_{i,\beta}^{\otimes 2})_{l,k}\right] \\
% & = \sum_{j,k \in \{1,\dots, d\}} (\Lambda_\beta)_{j,k} \sum_{l \in \{1, \dots, d\} }(\Lambda_\beta)_{j,l} \EE\left[\sum_{i=1}^n (D_{i,\beta}^{\otimes 2})_{l,k}\right] \\
% & =  \sum_{j,k \in \{1,\dots, d\}} (\Lambda_\beta)_{j,k} (I_d)_{j,k} \\
% & = Tr(\Lambda_\beta)
% \end{align*}
% and, since the eigenvalues of $\Lambda_\beta$ must be greater than $1$,
% \beq
% \label{eq:beta2bound}
% \beta_2 \geq d
% \eeq

\subsection{Large times}

Finally, we are left with bounding $\|\rho_t\|_p$ for "large" values of $t$ by using large-scale features of the $(W_i)_{1 \leq i \leq n}$. In practice, we improve on Proposition 6.1 \cite{ModerateDeviations}. However, while this result was derived in the general exchangeable pairs framework, our improvements require dealing with sums of independent random variables and thus to the Central Limit Theorem case. 

\begin{proposition}
\label{pro:Psi3}
Suppose $W = \sum_{i=1}^n W_i$ where the $(W_i)_{1 \leq i \leq n}$ are centered independent random variables on $\mathbb{R}^d$ with finite moment of order $p+2$ such that $\EE[W^{\otimes 2}] = I_d$. There exists $C > 0$ such that for any $p < q \leq p+2$ and $r$ verifying $\frac{1}{q} + \frac{1}{r} = \frac{1}{p}$ and any $t$ such that $\Delta(t) > (p-1) \max_{i \in \{1,\dots,n\}} \|\EE[W_i^{\otimes 2}]\|$, we have 
\begin{multline*}
\|\rho_t\|_p \leq \Psi_3(t) \coloneqq \frac{e^{-3t} \|\EE[W^{\otimes 3}] H_2(Z)\|_p }{2} 
+  \frac{C \|\EE[W^{\otimes 3}]\| W_{q}(\nu,\gamma)}{\eta_{2r}(t)^{3/2}}  \\
+ C \left(\sqrt{\frac{ p N_4 (t)}{ \eta_p(t) }} 
+ p \left(\frac{ N_{p+2}(t)}{\eta_p(t)}\right)^{1/p} \right) 
+ \frac{N'_4(t)}{\eta_p(t)},
\end{multline*} 
where
\begin{itemize}
\item $\eta_p(t) = \frac{\Delta(t)}{p-1}$;
\item $\xi_i(t) = \log\left(\frac{\eta_p(t)}{\|\EE[W_i^{\otimes 2}]\|} \right)$;
\item $\forall q > 2, N_q(t) = \sum_{i=1}^n \frac{\EE[\|D_{i}\|^{q} (1_{\|D_i\|^2 \geq \eta_p(t) \xi_i(t)} + \xi_i(t)^{-1})^{q/2-1}]}{\xi_i(t)}$;
\item $N'_4(t) = \sum_{i=1}^n \frac{\|\EE[D_i^{\otimes 2} \|W_i\|\|D_i\|]\|}{\sqrt{\|\EE[W_i^{\otimes 2}]\|} \xi_i(t)^{3/2}}$.
\end{itemize}
\end{proposition}

For any $i \in \{1,\dots,n\}$ and any $t > 0$, let $D_{i,t} = D_i 1_{\|D_i\|^2 \leq \eta_p(t) \xi_i(t)  }$.
Let us first rewrite $\rho_t$ with the help of the following result.

\begin{lemma} 
\label{lem:taut}
For any $i \in \{1,\dots,n\}$, the quantity
\[
\tau_{i,t} = \EE\left[D_{i,t}  \mid F_t\right] 
+ \sum_{k=1}^\infty \frac{\EE[(W'_i \otimes D_{i,t}^{\otimes k}) H_k(Z) \mid F_t]}{k! \Delta(t)^{k/2}} 
\]
verifies
\[
\EE[\tau_{i,t} \mid F_t] = 0.
\]
\end{lemma}

\begin{proof}
Let  $i \in \{1,\dots,n\}$ and let $\test$ be a smooth test function. 
Since $\Phi :x \rightarrow \EE[\phi(e^{-t} x + \sqrt{1-e^{-2t}}Z)]$ is real analytic (see e.g. Lemma~1 \cite{Bonis} or Lemma~6.4 \cite{ModerateDeviations}), we have 
\begin{align*}
\EE[ W'_i \test(F_t + e^{-t} D_{i,t})  1_{\|D_i\|^2 \leq \eta_p(t) \xi_i(t) }] 
 & = \sum_{k=0}^\infty \frac{e^{-kt}}{k!} \EE\left[ W'_i  \left<D_{i,t}^{\otimes k}, \nabla^{k} \test(F_t)\right> \right] \\
 & = \sum_{k=0}^\infty \frac{e^{-kt}}{k!} \EE\left[ (W'_i  \otimes D_{i,t}^{\otimes k}) \nabla^{k} \test(F_t) \right].
\end{align*}
Thus, by performing multiple integrations by parts with respect to the Gaussian measure (see e.g. Equation (16) \cite{Bonis}), we obtain 
\[
 \EE[ W'_i  \test(F_t + e^{-t} D_{i,t}) 1_{\|D_i\|^2 \leq \eta_p(t) \xi_i(t) }] 
 = \sum_{k=0}^\infty \frac{\EE\left[ (W'_i  \otimes D_{i,t}^{\otimes k}) H_k(Z) \test(F_t) \right]}{k! \Delta(t)^{k/2}}.
\]
Finally, since $W_i$ and $W'_i$ are independent and identically distributed, we have 
\[
\EE[ W'_i \test(F_t + e^{-t} D_{i,t})  1_{\|D_i\|^2 \leq \eta_p(t) \xi_i(t) }] = \EE[ W_i \test(F_t)  1_{\|D_i\|^2 \leq \eta_p(t) \xi_i(t) }]
\]
concluding the proof. 
\end{proof}

We are now ready to start the proof of Proposition~\ref{pro:Psi3}. Using Lemma~\ref{lem:taut}, we obtain 
\[
\rho_t = \rho_t + e^{-t} \sum_{i=1}^n \EE[\tau_{i,t} \mid F_t].
\]
Then, since $\sum_{i=1}^n \EE[W_i^{\otimes 2}] =  I_d$ and $\sum_{i=1}^n \EE[W_i^{\otimes 3}] = \EE[W^{\otimes 3}]$, we can write 
\[\rho_{t} = \frac{e^{-t} \EE[W^{\otimes 3}] \EE[H_2(Z) \mid F_t]}{2 \Delta(t)} + e^{-t} \sum_{i=1}^n \EE\left[W_i - \frac{\EE[W_i^{\otimes 2}]}{\sqrt{\Delta(t)}} Z - \frac{\EE[W_i^{\otimes 3}]}{2 \Delta(t)} H_2(Z)
+ \tau_{i,t} \mid F_t \right].
\]
Thus, combining the triangle inequality, Jensen's inequality and  Lemma~\ref{lem:HermiteNorm}, we obtain
\[
 \|\rho_t\|_p \leq \frac{e^{-t}| \|\EE[\EE[W^{\otimes 3}] H_2(Z) \mid F_t]\|_p}{2 \Delta(t)} + e^{-t} A(t),
\]
where 
\[
A(t) = \EE\left[ \left(\sum_{k=0}^\infty \left\|\sum_{i=1}^n A_{k,i}  \right\|^2 \right)^{p/2} \right]^{1/p} 
\]
with
\begin{itemize}
\item $A_{0,i} = \EE[D_{i,t} + W_i \mid W_i]$;
\item $\forall k \in \{1,2\}, A_{k,i} = \frac{ \EE[W'_i \otimes D_{i,t}^{\otimes k} \mid W_i] - \EE[W_i^{\otimes (k+1)}]}{\sqrt{k!} \eta_p(t)^{k/2}}$;
\item $\forall k > 2, A_{k,i} =  \frac{  \EE[W'_i \otimes D_{i,t}^{\otimes k} \mid W_i]}{\sqrt{k!} \eta_p(t)^{k/2}}$.
\end{itemize}
First, by Lemmas~\ref{lem:condExpec} and \ref{lem:condExp2}, we have 
\[
\frac{e^{-t}| \|\EE[\EE[W^{\otimes 3}] H_2(Z) \mid F_t]\|_p}{2 \Delta(t)} \leq \frac{e^{-3t}}{2} \|\EE[W^{\otimes 3}] H_2(Z)\|_p 
+  \frac{C \|\EE[W^{\otimes 3}]\| W_{q}(\nu,\gamma)}{\eta_{2r}(t)^{3/2}}
\]
where $q > p$ and $r$ is such that 
\[
\frac{1}{q} + \frac{1}{r} = \frac{1}{p}.
\]
% We are thus left with bounding $A(t)$. First, remark that, since $W'_i$ is an independent copy of $W_i$, we have 
% \[
% \EE[D_{i,t} + W \mid W] = -\EE[\bar{D}_{i,t} \mid W]
% \]
% and, similarly,
% \[
% \EE[W'_i \otimes D_{i,t} \mid W] - \EE[W_i^{\otimes 2}] = -\EE[W'_i \otimes \bar{D}_{i,t} \mid W].
% \]
% Therefore  \arprendre
% \[
% A(t) =  \left\|\left\|\sum_{i=1}^n  \EE[\bar{D}_{i,t} \mid W]\right\|^2 + \sum_{k=1}^2 \left\|\sum_{i=1}^n \frac{ \EE[W'_i \otimes \bar{D}_{i,t}^{\otimes k} \mid W]}{\sqrt{k!} \eta_p(t)^{k/2}} \right\|^2 + \sum_{k=3}^\infty \left\|\sum_{i=1}^n \frac{  \EE[W'_i \otimes D_{i,t}^{\otimes k} \mid W]}{\sqrt{k!} \eta_p(t)^{k/2}}\right\|^2\right\|^{1/2}_{p/2}.
% \]
Let $\bar{D}_{i,t} = D_i 1_{\|D_i\|^2 \geq \eta_p(t) \xi_i(t) }$.
In order to deal with $A(t)$, let us first remark that, since $\EE[W_i] = 0$ and since $W'_i$ and $W_i$ are independent, we have 
\[
A_{0,i} = \EE[D_{i,t} + W_i \mid W_i] = \EE[D_{i,t} - D_i \mid W_i] = -\EE[\bar{D}_{i,t} \mid W_i].
\]
And similarly,
\[
A_{1,i} = -\frac{\EE[W'_i \otimes \bar{D}_{i,t} \mid W_i]}{\sqrt{\eta_p(t)}}.
\]
Let us also note that 
\[
\EE[A_{2,i}] = - \frac{\EE[W'_i \otimes \bar{D}_{i,t}^{\otimes 2}]}{ 2 \eta_p(t)}.
\]
Then, viewing $A(t)$ as the $p$-norm of an infinite-dimensional vector, we can apply Rosenthal's inequality (see Lemma~\ref{lem:Rosenthal}) to obtain
\begin{multline*}
 \|\rho_t\|_p \leq  \sum_{i=1}^n \left( \sum_{k=0}^\infty \|\EE[A_{k,i}]\|^2 \right)^{1/2} + C\sqrt{p} \left(\sum_{i=1}^n \EE\left[\sum_{k=0}^\infty \|B_{k,i}\|^2 \right] \right)^{1/2} \\ + Cp \left(\sum_{i=1}^n \EE\left[\left(\sum_{k=0}^\infty \|B_{k,i}\|^2 \right)^{p/2} \right] \right)^{1/p},
\end{multline*}
where
\[
\forall k \in \mathbb{N}, 1 \leq i \leq n, B_{k,i} = \begin{cases}
A_{k,i} \text{ if $k \neq 2$} \\
 \frac{ \EE[W'_i \otimes D_{i,t}^{\otimes 2} \mid W_i]}{\sqrt{2} \eta_p(t)} \text{ if $k=2$}.
\end{cases}
\]
Let us conclude the proof by bounding these quantities. 

\subsubsection{Bounding the first term}
Let $i \in \{1, \dots, n\}$. 
First, let us note that since $W_i$ and $W'_i$ are independent,
\[
\EE[A_{0,i}] = \EE[\bar{D}_{i,t}] = 0.
\]
On the other hand, since $\EE[W_i] = 0$ and $\EE[W_i^{\otimes 2}] = I_d$, we have $\EE[W'_i \otimes \bar{D}_{i,t}] = - \EE[W_i \otimes \bar{D}_{i,t}]$. Hence, 
\[
 \EE[W'_i \otimes \bar{D}_{i,t}] = \frac{\EE[\bar{D}_{i,t}^{\otimes 2}]}{2}
\]
and 
\[
\|\EE[A_{1,i}]\|^2 = \frac{\|\EE[\bar{D}_{i,t}^{\otimes 2}]\|^2}{4 \eta_p(t)} \leq \frac{\|\EE[D_{i}^{\otimes 2} \|D_{i}\|^2]\|^2}{4 \eta_p(t)^{3} \xi_i(t)^2 } \leq \frac{\|\EE[D_{i}^{\otimes 2} \|D_{i}\| \|W_i\|]\|^2}{(\eta_p(t) \xi_i(t))^{3}} \xi_i(t).
\]
On the other hand, 
\[
\|\EE[A_{2,i}]\|^2 \leq \frac{\|\EE[\bar{D}_{i,t}^{\otimes 2} \|W_i\|]\|^2}{\eta_p(t)^2} 
 \leq \frac{\|\EE[D_{i}^{\otimes 2} \|D_{i}\| \|W_i\|]\|^2}{(\eta_p(t) \xi_i(t))^{3}} \frac{\xi_i(t)^2}{2}.
\]
Finally, for any $k \geq 3$,
\[
\|\EE[A_{k,i}]\|^2 \leq \frac{\|\EE[D_{i,t}^{\otimes 2} \|D_{i,t}\|^{k-2} \|W_i\|]\|^2}{k! \eta_p(t)^k}  \leq \frac{\|\EE[D_{i}^{\otimes 2} \|D_{i}\| \|W_i\|]\|^2}{(\eta_p(t) \xi_i(t))^{3}} \frac{\xi_i(t)^k}{k!}.
\]
Therefore, by definition of $\xi_i(t)$,
\begin{align*}
\sum_{k=0}^\infty \|\EE[A_{k,i}]\|^2
& \leq \frac{\|\EE[D_i^{\otimes 2} \|D_i\| \|W_i\|]\|^2}{(\eta_p(t) \xi_i(t))^3} \sum_{k=1}^\infty \frac{\xi_i(t)^{k}}{k!} \\
& \leq \frac{\|\EE[D_i^{\otimes 2} \|D_i\| \|W_i\|]\|^2}{(\eta_p(t) \xi_i(t))^3} e^{\xi_i(t)} \\
& \leq \frac{\|\EE[D_i^{\otimes 2} \|D_i\| \|W_i\|]\|^2}{\|\EE[W_i^{\otimes 2}]\|\eta_p(t)^2 \xi_i(t)^3}
\end{align*}
and thus
\[
 \sum_{i=1}^n \left( \sum_{k=0}^\infty \|\EE[A_{k,i}]\|^2 \right)^{1/2} \leq \frac{N'_4(t)}{\eta_p(t)}.
\]

\subsubsection{Bounding the last two terms}
Let $i \in \{1, \dots, n\}, q \in [2,p]$.
First, by Jensen' inequality and by definition of $\bar{D}_{i,t}$, 
\[ 
\|B_{0,i} \|^2 \leq \|\EE[\bar{D}_{i,t} \mid W_i]\|^2 \leq \EE[\|\bar{D}_{i,t}\|^2 \mid W_i] \leq \frac{\EE[\|\bar{D}_{i,t}\|^{2+4/q} \mid W_i]}{(\eta_p(t) \xi_i(t) )^{2/q}}.
\]
Let $W''_i$ and $\bar{D}'_{i,t}$ be a conditionally independent copies of $W'_i$ and $\bar{D}_{i,t}$ with respect to $W_i$. We have  
\begin{align*}
 \|\EE[W'_i \otimes \bar{D}_{i,t} \mid W_i]\|^2 & =\EE[\left< W'_i  \otimes \bar{D}_{i,t}, W''_i  \otimes \bar{D}'_{i,t} \right> \mid W_i] \\
 &  =\EE[\left< W'_i, W''_i \right> \left<\bar{D}_{i,t}, \bar{D}'_{i,t} \right> \mid W_i]
\end{align*}
and, by Cauchy-Schwarz's inequality,
\begin{align*}
 \|\EE[W'_i \otimes \bar{D}_{i,t} \mid W_i]\|^2 & \leq \EE[\left< W'_i, W''_i \right>^2 \mid W_i]^{1/2} \EE[\left<\bar{D}_{i,t},  \bar{D}'_{i,t} \right>^2 \mid W_i]^{1/2} \\
 &  \leq \EE[\left< W_i^{\prime \otimes 2}, W_i^{\prime \prime \otimes 2} \right> \mid W_i]^{1/2} \EE[\left<\bar{D}_{i,t}^{\otimes 2}, \bar{D}_{i,t}^{\prime \otimes 2} \right> \mid W_i]^{1/2} \\
 & \leq \|\EE[W_i^{\prime \otimes 2} \mid W_i]\| \|\EE[\bar{D}_{i,t}^{\otimes 2} \mid W_i]\|.
\end{align*}
Since $W'_i$ is independent of $W$, $\|\EE[W_i^{\prime \otimes 2} \mid W_i]\| = \|\EE[W_i^{\prime \otimes 2}]\| = \|\EE[W_i^{\otimes 2}]\|$. Thus, 
\begin{align*}
\|B_{1,i}\|^2 & = \frac{\|\EE[W'_i \otimes \bar{D}_{i,t} \mid W_i]\|^2}{\eta_p(t)} \\
& \leq \frac{\|\EE[W_i^{\otimes 2}]\| \EE[\|\bar{D}_{i,t}\|^2 \mid W_i]}{\eta_p(t)}\\
& \leq \frac{\|\EE[W_i^{\otimes 2}]\| \EE[\|D_{i}\|^{2+4 / q}\mid W_i] }{(\eta_p(t) \xi_i(t) )^{1 + 2/q}} \xi_i(t).
 \end{align*}
 On the other hand, 
 \begin{align*}
\|B_{2,i}\|^2 & \leq  \frac{\|\EE[W'_i \otimes D_{i,t}^{\otimes 2} \mid W_i]\|^2}{2\eta_p(t)^2} \\
& \leq \frac{\|\EE[W_i^{\otimes 2}]\| \EE[\|D_{i,t}\|^4 \mid W_i]  }{2\eta_p(t)^2}\\
& \leq \frac{\|\EE[W_i^{\otimes 2}]\| \EE[\|D_{i}\|^{2+4 / q}\mid W_i] }{(\eta_p(t) \xi_i(t) )^{1 + 2/q}} \frac{\xi^2_i(t)}{2}.
 \end{align*}
 Finally, for any $k \geq 3$,
 \[
 \|B_{k,i}\|^2 \leq \frac{\|\EE[W_i^{\otimes 2}]\|  \EE[\|D_{i,t}\|^{2k} \mid W_i]}{k! \eta_p(t)^k}  \leq \frac{\|\EE[W_i^{\otimes 2}]\| \EE[\|D_{i}\|^{2+4 / q}\mid W_i] }{(\eta_p(t) \xi_i(t) )^{1 + 2/q}} \frac{\xi_i(t)^k}{k!}.
 \]
Combining both these bounds yields, by definition of $\xi_i(t)$
\begin{align*}
\sum_{k=1}^\infty \|B_{k,i}\|^2
& \leq \frac{\|\EE[W_i^{\otimes 2}]\| \EE[\|D_{i}\|^{2 + 4 / q}\mid W_i]}{(\eta_p(t) \xi_i(t) )^{1 + 2 / q}} \sum_{k=1}^\infty \frac{\xi_i(t)}{k!} \\
& \leq \frac{\|\EE[W_i^{\otimes 2}]\| \EE[\|D_{i}\|^{2 + 4 / q}\mid W_i]}{(\eta_p(t) \xi_i(t) )^{1+2 / q}} e^{\xi_i(t)} \\
& \leq \frac{\EE[\|D_{i}\|^{2 +4 / q}\mid W_i] }{\xi_i(t)^{1+2 / q} \eta_p(t)^{2/q}}
\end{align*}
Thus,
\[
\sum_{k=0}^\infty \|B_{k,i}\|^2
\leq \frac{\EE[\|D_{i}\|^{2 + 4/q} (1_{\|D_i\|^2 \geq \eta_p(t) \xi_i(t)} + \xi_i(t)^{-1}) \mid W_i]}{(\eta_p(t) \xi_i(t) )^{2/q}}
\]
and, by Jensen's inequality, 
\[
\EE\left[\left(\sum_{k=0}^\infty \|B_{k,i}\|^2 \right)^{q/2} \right]  
 \leq \frac{\EE[\|D_{i}\|^{q+2} (1_{\|D_i\|^2 \geq \eta_p(t) \xi_i(t)} + \xi_i(t)^{-1})^{q/2}]}{\eta_p(t) \xi_i(t) }.
\]
Therefore,
\[
\left(\sum_{i=1}^n \EE\left[\left(\sum_{k=0}^\infty \|B_{k,i}\|^2 \right)^{q/2} \right] \right)^{1/q} \leq \left( \frac{N_{q+2}(t)}{\eta_p(t)} \right)^{1/q}.
\]

\subsection{Combining times}

We are now ready to conclude the proof of Theorem~\ref{thm:mainCLT}. Let $\epsilon_1$ and $\epsilon_2$ such that $\eta_p(\epsilon_1) \coloneqq \frac{\Delta(\epsilon_1)}{p-1} = \gap^2$ and $\epsilon_2$ be such that $\eta_p(\epsilon_2) \coloneqq \frac{\Delta(\epsilon_2)}{p-1} = \epsilon \coloneqq \max_{i \in \{1, \dots, n\}} \|\EE[W_i^{\otimes 2}]\|^{2/3}$. Remark that, by assumption, $\epsilon_1 < \epsilon_2$. In the following computations, we will rely on the fact that $\Delta(t) \geq 2t$. By (\ref{eq:small-time}) and Propositions~\ref{pro:Psi2} and \ref{pro:Psi3} we have 
\[
\int_0^\infty \|\rho_t\|_p \, dt \leq \int_0^{\epsilon_1} \Psi_1(t) \, dt  
+ \int_{\epsilon_1}^{\epsilon_2} \Psi_2(t) \, dt 
+ \int_{\epsilon_2}^\infty \Psi_3(t) \, dt. 
 \]
First, since $\beta^2 < \epsilon$
\begin{align*}
\int_0^{\epsilon_1} \Psi_1(t) \, dt & \leq C \left(p \sqrt{\eta_p(\epsilon_1) d } +  p \eta_p(\epsilon_1) \left( \sqrt{p d} + p L_p^{1/p} \right) \right) \\
& \leq C \left(p \gap \sqrt{d } +  p \epsilon \left( \sqrt{pd} + p L_p^{1/p} \right) \right).
\end{align*}
Then,
\begin{align*}
\int_{\epsilon_1}^{\epsilon_2} \Psi_2(t) \, dt & \leq C p \eta_p(\epsilon_2) \left( \sqrt{p}( \sqrt{ \beta_2} + \sqrt{d}) + p  \left(\beta_p^{1/p} + L_p^{1/p}\right) \right) + \frac{Cp\beta^2\sqrt{d}}{\sqrt{\eta_p(\epsilon_1)}} \\
& \leq  C p \epsilon \left( \sqrt{p}( \sqrt{ \beta_2} + \sqrt{d}) + p  \left(\beta_p^{1/p} + L_p^{1/p}\right) \right) + C p \gap \sqrt{d}.
\end{align*}
Finally, let $t \geq \epsilon_2$. Since $\eta_p(t) \geq \max_{i \in \{1, \dots, n\}} \|\EE[W_i^{\otimes 2}]\|^{2/3}$, we have, for any $i \in \{1, \dots, n\}, \xi_i(t) \geq C \xi_i$. On the other hand, since $p \leq \min_{i \in \{1, \dots, n\}} \|\EE[W_i^{\otimes 2}]\|^{-1}$, we have $-\log(\epsilon_2) \leq -\log(\eta_p(\epsilon_2)) - \log(p-1) \leq C \xi_i$. Hence, since $N_4(.), N_{p+2}(.)$ and $N'_4(.)$ are decreasing functions, 
\begin{align*}
\int_{\epsilon_2}^{\infty} \Psi_3(t) \, dt &   \leq \frac{\|\EE[W^{\otimes 3}] H_2(Z)\|_p }{6} + \frac{ C(2r-1)^{3/2} \|\EE[W^{\otimes 3}]\| W_{q}(\nu,\gamma)}{\sqrt{(p-1) \eta_p(\epsilon_2)}}
\\
& \hspace{3cm} +   C p \left(  
  \sqrt{N_4(\epsilon_2)} 
+ (p N_{p+2}(\epsilon_2))^{1/p}
-  N'_4(\epsilon_2)\log(\epsilon_2)  \right) \\
& \leq \frac{\|\EE[W^{\otimes 3}] H_2(Z)\|_p }{6} + \frac{C (2r-1)^{3/2} \|\EE[W^{\otimes 3}]\| W_{q}(\nu,\gamma)}{\sqrt{(p-1)\epsilon}} \\
&  \hspace{3cm} +  C p \left(   
  \sqrt{N_4} 
+ (p N_{p+2})^{1/p}
+  N'_4 \right).
\end{align*}
which concludes the proof of Theorem~\ref{thm:mainCLT}.

\section{Technical lemmas}

\begin{lemma}[Rosenthal inequality, Theorem 5.2 \cite{Rosenthal}]
\label{lem:Rosenthal}
There exists $C > 0$ such that, for any $p \geq 2$ and any independent random variables $(U_i)_{1 \leq i \leq n}$  with finite moment of order $p$ taking values in a Hilbert space $\mathcal{H}$, we have  
\[
\left\|\sum_{i=1}^n U_i - \EE[U_i] \right\|_{\mathcal{H},p} \leq C\sqrt{p}\left(\sum_{i=1}^n \|U_i\|_{\mathcal{H},2}^2 \right)^{1/2} + C p \left(\sum_{i=1}^n \|U_i\|_{\mathcal{H},p}^p \right)^{1/p},
\]
where, for any random variable $X$ taking values in $\mathcal{H}$ and any $q > 0$,
\[
\|X\|_{\mathcal{H},q} = \EE[\|X\|^q_{\mathcal{H}}].
\]
\end{lemma}

\begin{proof}
By Theorem 5.2 \cite{Rosenthal},
\[
\left\|\sum_{i=1}^n U_i - \EE[U_i] \right\|_{\mathcal{H},p} \leq C\sqrt{p}\left(\sum_{i=1}^n \|U_i - \EE[U_i]\|_{\mathcal{H},2}^2 \right)^{1/2} + C p \left(\sum_{i=1}^n \|U_i - \EE[U_i]\|_{\mathcal{H},p}^p \right)^{1/p}.
\]
Now for $q \in [2,p]$, combining the triangle and Jensen's inequalities yields 
\[
\|U_i - \EE[U_i]\|_{\mathcal{H},q} \leq \|U_i\|_{\mathcal{H},q} + \|\EE[U_i]\|_{\mathcal{H}} \leq 2 \|U_i\|_{\mathcal{H},q},
\]
concluding the proof.

\end{proof}

% \begin{lemma}
% Let $k \geq 3$, we have 
% \[
% \|\EE[W'_i \otimes D_{i}^{\otimes k} \mid W_i]\|_p^p \leq (k-3) (n \eta_p(t))^{(k-3)/2} \|\EE[W_i^{'\otimes 2} \otimes D_{i}^{\otimes 2} \mid W_i]\|.
% \]
% \end{lemma}

% \begin{proof}
% We have 
% \[
% \EE[W'_i \otimes D_{i}^{\otimes k} \mid W_i] = \EE[W_i^{'\otimes 2} \otimes D_{i}^{\otimes (k-1)} \mid W_i] + W_i \EE[W'_i \otimes D_{i}^{\otimes (k-1)} \mid W_i].
% \]
% Since
% \[
% \|\EE[W_i^{'\otimes 2} \otimes D_{i}^{\otimes (k-1)} \mid W_i]\| \leq (n \eta_p(t))^{(k-3)/2} \|\EE[W_i^{'\otimes 2} \otimes D_{i}^{\otimes 2} \mid W_i]\|
% \]
% and 
% \[
% \|W_i \EE[W'_i \otimes D_{i}^{\otimes (k-1)} \mid W_i]\| \leq \sqrt{n \infnorm} \|\EE[W'_i \otimes D_{i}^{\otimes (k-1)} \mid W_i]\|,
% \]
% using the triangle inequality yields 
% \[
% \|\EE[W'_i \otimes D_{i}^{\otimes k} \mid W_i]\| \leq (k-3) (n \infnorm)^{(k-3)/2} \|\EE[W_i^{'\otimes 2} \otimes D_{i}^{\otimes 2} \mid W_i]\| + \|\EE[W'_i \otimes W_i \otimes (W'_i - W_i)]\|.
% \]
% Finally, since $W_i$ and $W'_i$ are independent and centered, 
% \[
% \EE[W'_i \otimes W_i \otimes (W'_i - W_i)] = 0.
% \]
% \end{proof}

% \begin{lemma}[Equation (16) \cite{Bonis}]
% \label{lem:ipp}
% Let $Z$ be a random variable and $\phi$ a smooth function with compact support. Then, 
% \[
% \EE[\nabla^i \phi(Z)] = \EE[H_i(Z) \phi(Z)]. 
% \]
% \end{lemma}

\begin{lemma}[Lemma 3 \cite{Bonis}]
\label{lem:HermiteNorm}
 Let $Z$ be a $d$-dimensional standard normal random variable. For any $p \geq 2$, $ k \in \mathbb{N}$ and $M \in (\RR^d)^{\otimes k+1}$, we have   
 \[
 \|M H_k(Z)\|_p^2 \leq   (p-1)^{k} k! \|M\|^2. 
 \]
 \end{lemma}

\begin{lemma} 
\label{lem:condExpec}
Let $X,Y$ and $Z$ be three random variables on $\mathbb{R}^d$ such that $Z$ is drawn from the Gaussian measure $\gamma$ and is independent from $(X,Y)$. Let $q > p \geq 2$ and suppose that $X$ and $Y$ have finite moment of order $q$. Then, for any $k \geq 0$ and any $i \in \{1,\dots,d\}^k$,
\[
\|\EE[H_i(Z) \mid X + Z] - \EE[H_i(Z) \mid Y + Z]\|_p  \leq C \sqrt{(2r-1)^{k+1} (k+1)!} \|Y - X\|_q,
\]
where $C > 0$ is a generic constant, $H_i = (H_k)_i$ and $r$ is such that $\frac{1}{r} + \frac{1}{q} = \frac{1}{p}$. 
\end{lemma}

\begin{proof}
Let $\epsilon = Y- X$. 
We have
\[
\EE[H_i(Z) \mid X + Z] = \EE[H_i(Z) \mid Y + Z + \epsilon]
\]
and 
\[
\EE[H_i(Z) \mid Y + Z] - \EE[H_i(Z) \mid X + Z] = \int_0^1 \frac{d}{dt} \EE[H_i(Z) \mid Y + Z + t \epsilon] \, dt.
\]
Let us denote the density of $Z$ by $f_\gamma$ and the measure of $(Y, \epsilon)$ by $\mu$. For any $t \in [0,1]$, let 
\[
f(t) = \int (-1)^k \nabla_i f_\gamma\left(Z + Y-y' + t( \epsilon - \epsilon')\right)  d \mu(y',\epsilon'),
\]
where $\nabla_i \cdot = (\nabla^k \cdot)_i$. 
We then have 
\[
f'(t) = \int \left< \epsilon - \epsilon', (-1)^k \nabla \nabla_{i} f_\gamma\left(Z + Y-y' + t( \epsilon - \epsilon')\right)  \right> d \mu(y',\epsilon').
\]
Similarly, letting
\[
g(t) = \int f_\gamma\left(Z + Y-y' + t( \epsilon - \epsilon')\right) d \mu(y',\epsilon'),
\]
we have 
\[
g'(t) = \int \left< \epsilon - \epsilon', \nabla f_\gamma \left(Z + Y-y' + t( \epsilon - \epsilon')\right) \right> d \mu(y',\epsilon').
\]
By definition of the conditional expectation,
\begin{itemize}
\item $ \frac{f(t)}{g(t)} = \EE[H_i(Z) \mid Y + Z + t \epsilon]$; 
\item $\frac{g'(t)}{g(t)} = \EE[\left< \epsilon, Z\right > \mid Y + Z + t \epsilon] - \left< \epsilon, \EE[Z \mid Y + Z + t \epsilon] \right> $ and 
\item $\frac{f'(t)}{g(t)} = \EE[\left< \epsilon, H_{i+1}(Z)\right > \mid Y + Z + t \epsilon] - \left< \epsilon, \EE[H_{i+1}(Z) \mid Y + Z + t \epsilon] \right>$,
\end{itemize}
where  $H_{i+1}(x) = (-1)^{k+1}  \frac{\nabla \nabla_i f_\gamma(x)}{f_\gamma(x)}$.
Therefore, letting $G_t =  Y + Z + t \epsilon$, we obtain 
\begin{multline*}
\frac{d}{dt} \EE[H_i(Z) \mid G_t]  = \left(\frac{f}{g}\right)'  (t) 
 = \left<\epsilon,\EE[Z \mid G_t] \EE[H_i(Z) \mid G_t] - \EE[H_{i+1}(Z) \mid G_t] \right> \\ - \left(\EE[\left<\epsilon, Z\right> \mid G_t] \EE[H_i(Z) \mid G_t] - \EE[\left< \epsilon, H_{i+1}(Z) \right> \mid G_t]\right).
\end{multline*}
Applying the triangle inequality along with Cauchy-Schwarz, Holder's and Jensen's inequalities then yields
\[
\left\|\frac{d}{dt} \EE[H_i(Z) \mid G_t] \right\|_p \leq C \|\epsilon\|_q  (\|H_{i+1}(Z)\|_r + \|Z\|_{2r} \|H_i(Z)\|_{2r}), 
\]
where
\[
\frac{1}{q} + \frac{1}{r} = \frac{1}{p}. 
\]
Finally applying Lemma~\ref{lem:HermiteNorm} yields
\[
\left\|\frac{d}{dt} \EE[H_i(Z) \mid Y + Z + t \epsilon] \right\|_p \leq C \sqrt{ (2r-1)^{(k+1)}  (k+1)!} \|\epsilon\|_q,
\]
concluding the proof.
\end{proof}

\begin{lemma}
\label{lem:condExp2}
    Let $Y$ and $Z$ be two independent standard normal random variables on $\mathbb{R}^d$. Then, for any $k \geq 1$ and any $\alpha > 0$, we have 
    \[
    \EE[H_k(Z) \mid \alpha Y + \sqrt{1 - \alpha^2} Z] = (1 - \alpha^2)^{k/2} H_k(\alpha Y + \sqrt{1 - \alpha^2} Z).
    \]
\end{lemma}

\begin{proof}
Let $\test$ be a smooth function with compact support. By performing multiple integrations by parts with respect to the Gaussian measure (see e.g. Equation (16) \cite{Bonis}), we obtain  
\begin{align*}
\EE\big[\EE[H_k(Z) \mid \alpha Y + \sqrt{1 - \alpha^2} Z] &\test(\alpha Y + \sqrt{1 - \alpha^2} Z)\big] \\
& = \EE[H_k(Z) \test(\alpha Y + \sqrt{1 - \alpha^2} Z)] \\
& = (1 - \alpha^2)^{k/2} \EE[\nabla^k \test(\alpha Y + \sqrt{1 - \alpha^2} Z)] \\
& = (1 - \alpha^2)^{k/2} \EE[H_k(\alpha Y + \sqrt{1 - \alpha^2} Z) \test(\alpha Y + \sqrt{1 - \alpha^2} Z)],
\end{align*}
concluding the proof.
\end{proof}

\begin{lemma}
    \label{lem:CS}
    Let $X$ be a random variable on $\mathbb{R}^d$ with identity covariance matrix. Then, 
    \[
    \|\EE[X^{\otimes 3}]\|^2 \leq  \sqrt{d}  \|\EE[X^{\otimes 2} \|X\|^2]\|.
    \]
\end{lemma}

\begin{proof}
Let $X'$ be an independent copy of $X$. By Cauchy-Schwarz's inequality, we have 
\begin{align*}
\|\EE[X^{\otimes 3}]\|^2 & = \EE\left[ \left<X^{\otimes 3}, X^{\prime \otimes 3} \right> \right] \\
& = \EE\left[ \left<X, X'\right>^3 \right] \\
& \leq \EE\left[ \left<X, X'\right>^2 \right]^{1/2} \EE\left[ \left<X, X'\right>^4 \right]^{1/2} \\
& \leq \EE\left[ \left<X^{\otimes 2}, X^{\prime \otimes 2}\right> \right]^{1/2} \EE\left[ \left<X^{\otimes 2}, X^{\prime \otimes 2}\right> \|X\|^2 \|X'\|^2 \right]^{1/2} \\
& \leq \|\EE[X^{\otimes 2}]\| \|\EE[X^{\otimes 2} \|X\|^2]\| \\
& \leq \|\EE[I_d]\| \|\EE[X^{\otimes 2} \|X\|^2]\| \\
& \leq \sqrt{d}  \|\EE[X^{\otimes 2} \|X\|^2]\|.
\end{align*}

\end{proof}

\begin{lemma}
\label{lem:third}
Let $X$ be a centered random variable on $\mathbb{R}^d$ with identity covariance matrix and suppose there exists $\tau : \mathbb{R}^d \rightarrow (\mathbb{R}^d)^{\otimes 2}$ such that, for any smooth test function $\test : \mathbb{R}^d \rightarrow \mathbb{R}^d $, 
\beq
\label{eq:steinKernel}
\EE[\left<X, \test(X)\right>] = \EE[\left<\tau(X), \nabla \test(X)\right>].
\eeq
Then
\[
\|\EE[X^{\otimes 3}]\| \leq 2 \EE[\|\tau(X) - I_d\|^2]^{1/2}.
\]
\end{lemma}

\begin{proof}
The proof follows the proof of Equation (3.14) \cite{klartag}, except we use (\ref{eq:steinKernel}) in place of the Poincaré inequality. 

Let $B = \EE[X^{\otimes 3}]$. We have 
\[
\|B\|^2 = \EE\left[\left<X, B X^{\otimes 2}\right>\right].
\]
By definition of $\tau$, we obtain 
\begin{align*}
\|B\|^2 & = \EE\left[\left<\tau(X), \nabla (B X^{\otimes 2})\right>\right] \\
& = 2 \EE\left[\left<\tau(X),  B X\right>\right],
\end{align*}
where
\[
(BX)_{i,j} = \sum_{k=1}^d B_{i,j,k} X_k.
\]
Since $X$ is centered, we have $\EE[B X] = 0$ and
\[
\|B\|^2 = 2 \EE\left[\left<\tau(X) - I_d,  B X\right>\right].
\]
Finally, by Cauchy-Schwarz's inequality and since $\EE[X^{\otimes 2}] = I_d$, 
\begin{align*}
\|B\|^2 &\leq  \EE[\|\tau(X) - I_d\|^2]^{1/2} \EE[\|B X\|^2]^{1/2} \\
& \leq 2\|B\| \EE[\|\tau(X) - I_d\|^2]^{1/2}, 
\end{align*}
concluding the proof.
\end{proof}

\bibliographystyle{amsplain}
\bibliography{biblio}

\end{document}